\documentclass{amsart}

\usepackage{graphicx}
\usepackage{amsmath}
\usepackage{amsfonts}
\usepackage{amssymb}
\usepackage{mathtools}
\usepackage{amscd}
\usepackage[all,cmtip]{xy}
\usepackage{amsthm}
\usepackage[capitalise]{cleveref}
\usepackage{xcolor}

\newtheorem{theorem}{Theorem}[section]
\newtheorem{proposition}[theorem]{Proposition}
\newtheorem{corollary}[theorem]{Corollary}
\newtheorem{lemma}[theorem]{Lemma}

\newtheorem{step}{Step}

  \theoremstyle{definition}

\newtheorem{remark}[theorem]{Remark}
\newtheorem{question}[theorem]{Question}

\newcommand{\dbZ}{\mathbb{Z}}

\newcommand{\Z}{\mathbb{Z}}

\newcommand{\calA}{{\mathcal A}}
\newcommand{\calB}{{\mathcal B}}

\newcommand{\calF}{{\mathcal F}}
\newcommand{\calG}{{\mathcal G}}

\newcommand{\calP}{{\mathcal P}}

\newcommand{\orf}[1]{O_\calF #1}

\newcommand{\nbeq}{\begin{equation}}
\newcommand{\neeq}{\end{equation}}
\newcommand{\beq}{\begin{equation*}}
\newcommand{\eeq}{\end{equation*}}

\newcommand{\fpn}{\mathrm{FP}_{n}}

\newcommand{\f}[1]{\mathrm{F}_{#1}}
\newcommand{\fn}{\mathrm{F}_{n}}

\newcommand{\esstriv}[2]{$\pi_{#1}$-$#2$-essentially trivial}
\newcommand{\hesstriv}[2]{$H_{#1}$-$#2$-essentially trivial}

\newcommand{\ffp}[1]{\calF \text{-}\mathrm{FP}_{#1}}
\newcommand{\gf}[1]{\calG \text{-}\mathrm{F}_{#1}}
\newcommand{\ffpn}{\calF\text{-}\mathrm{FP}_{n}}

\newcommand{\ff}[1]{\calF \text{-}\mathrm{F}_{#1}}
\newcommand{\ffn}{\calF\text{-}\mathrm{F}_{n}}

\DeclareMathOperator{\CAT}{CAT}

\begin{document}

\title{Brown's Criterion and classifying spaces for families}

\author{Eduardo Mart\'inez-Pedroza}
\address{Memorial University  St. John's, Newfoundland and Labrador, Canada}
\email{eduardo.martinez@mun.ca}

\author{Luis Jorge S\'anchez Salda\~na}
\address{Department of Mathematics, The Ohio State University, 100 Math Tower,
231 West 18th Avenue,
Columbus, OH 43210-1174, USA}
\email{luisjorge@ciencias.unam.mx}

\date{}




\begin{abstract}
Let $G$ be a group and $\calF$ be a family of subgroups closed under conjugation and subgroups. A model for the classifying space $E_\calF G$ is a $G$-CW-complex $X$ such that every isotropy group belongs to $\calF$, and for all $H\in \calF$ the fixed point subspace $X^H$ is contractible. The group $G$ is of type $\ffn$ if it admits a model for $E_\calF G$ with $n$-skeleton with compact orbit space. The main result of the article provides is a characterization of $\ff{n}$ analogue to Brown's criterion for $\mathrm{FP}_n$. As applications we provide criteria for this type of finiteness properties with respect to families to be preserved by finite extensions, a result that contrast with examples of Leary and Nucinkis. We also recover L\"uck's characterization of property $\underline{\mathrm{F}}_n$ in terms of the finiteness properties of the Weyl groups.  
\end{abstract}
\maketitle

\section{Introduction}

 Let $G$ be a group and $\calF$ be a family of  subgroups, i.e. a collection of subgroups of $G$ such that it is closed under conjugation and under taking subgroups. The family $\calF$ is called~\emph{trivial} if it only consists of the trivial subgroup.  In this article, a $G$-CW-complex is a CW-complex together with a cellular action of $G$ such that if a group element fixes a cell setwise, then it fixes the cell pointwise. 
 
 A model for the classifying space $E_\calF G$ is a $G$-CW-complex $X$ such that every isotropy group belongs to $\calF$, and for all $H\in \calF$ the fixed point subspace $X^H$ is contractible. A $G$-CW-complex $Y$ is an $\calF$-$G$-CW-complex if every isotropy group of $G$ belongs to $\calF$. A model for $E_\calF G$ can be equivalently defined as  $\calF$-$G$-CW-complex $X$, such that, for every $\calF$-$G$-CW-complex $Y$ there exists a $G$-map $Y\to X$, unique up to $G$-homotopy.
 
 Let $n\geq 0$. We say $G$ is of type $\ffn$ if it admits a model for $E_\calF G$ with $n$-skeleton with compact orbit space. Equivalently, $G$ is of type $\ffn$ if and only if there exists an $n$-dimensional $G$-CW-complex $Y$ such that:
 
 \begin{itemize}
     \item all of its isotropy groups, $G_y$ with $y\in Y$ belong to $\calF$,
     \item $Y^H$ is $(n-1)$-connected for all $H\in \calF$, and
     \item the orbit space $Y/G$ is compact (or $Y$ is $G$-finite).
 \end{itemize}
 In this situation we say that $Y$ is a \emph{$G$-witness for $\ff{n}$}.

 If $\calF$ contains only the trivial subgroup of $G$, then we recover the classical properties $F_n$ for discrete groups (see \cite{Br94}). 
 In the case that $\calF$ is the family of finite subgroups, then the notation $E_\calF G$ and $\ff{n}$ is replaced by $\underline{E}G$ and $\underline{\mathrm{F}}_n$.

 In~\cite{Br87} Brown proved a topological characterization, nowadays known as \textit{Brown's criterion}, for the algebraic counterpart of   property $\fn$, the property known as $\mathrm{FP}_n$. More rencently, Fluch and Witzel~\cite{FW13} proved an analogous characterization for the property $\ffpn$. In this article we prove a similar characterization dealing with property $\ffn$. 
 
 Brown's article proved {Brown's criteria} for properties $\mathrm{F_1}$ and $\mathrm{F_2}$ relying on their topological interpretations as being finitely generated and finitely presented respectively. A Brown's criterion for property $\mathrm{F}_n$, with $n>2$, can be proved using original's Brown's criteria for $\mathrm{F}_2$ and $\fpn$, and the fact that a group is $\mathrm{F}_n$ if and only if it is $F_2$ and $\mathrm{FP}_n$.  For a group $G$ and an arbitrary family $\mathcal{F}$, this strategy does not succeed to obtain a Brown's criterion for $\ffn$ due to the subtleness of properties $\ffn$ for $n=1,2$ in contrast with their  counterparts $\mathrm{F}_1$ and $\mathrm{F}_2$, see Section~\ref{sec:WitnessF0}.


 On the other hand, Dru\c tu and Kapovich proved in~\cite{DK18} a Brown's criterion for property ${F}_n$ using the Rips complex of $G$.  In the present paper we closely follow their strategies to prove our main theorem, replacing the role of the Rips complex, by a suitable classifying space.
 
Let $X$ be a $G$-CW-complex, let $\calF$ be a family, and let $n>0$. We say $X$ is \emph{$\calF$-$n$-good} if the following conditions are satisfied:

\begin{enumerate}
    \item For all $H\in \calF$, the fixed point set $X^H$ is
    non-empty, equivalently, $\pi_0(X^H)$ is a non-empty set.
    
    \item For all $H\in\calF$, and for all $x_0\in X^H$, $X$ is \emph{$\calF$-connected up to dimension $n-1$}, i.e. $\pi_0(X^H)$ is a set with exactly one element and for all $0< k\leq n-1$, $\pi_k(X^H,x_0)$ is the trivial group. 
    
    \item For every $p$-cell $\sigma$ of $X$, $p\leq n$,  the (pointwise) stabiliser $G_\sigma$ of $\sigma$ is of type $(\calF \cap G_\sigma )$-$\f{n-p}$.
\end{enumerate}
 
\begin{remark}\label{remark:Fngood}
Let $X$ be an $\calF$-$n$-good complex.
\begin{enumerate}
    \item Condition (1) implies that $\calF$ is contained in the family of $G$ generated by the isotropy groups of $X$, i.e. the smallest family of subgroups of $G$ which contains all the isotropy groups of $X$.
    \item If the family of $G$ generated by the isotropy groups of $X$ equals $\calF$, then $X^{(n)}$ is the $n$-skeleton of a model for $E_\calF G$.
    \item Observe that the notion of $\calF$-$n$-good only depends on the $n$-skeleton of $X$, i.e. $X$ is $\calF$-$n$-good if and only if $X^{(n)}$ is $\calF$-$n$-good.
\end{enumerate}
\end{remark}

By a \emph{filtration} of $X$ we mean  a family of $G$-subcomplexes $\{X_\alpha\}_{\alpha\in I}$, such that $I$ is a directed set, $X_\alpha \subseteq X_\beta $ if $\alpha \leq \beta$, and $X=\bigcup_\alpha X_\alpha$. A filtration $\{X_\alpha\}_{\alpha \in I}$ of $X$
is said to be of \emph{finite $n$-type} if the $n$-skeleta $X_\alpha^{(n)}$ is $G$-finite for every $\alpha \in I$.

Let $k$ be a positive integer. We say the filtration $\{X_\alpha\}_{\alpha \in I}$ is \textit{ $\pi_k$-$\calF$-essentially trivial} (resp. \textit{ $\pi_0$-$\calF$-essentially trivial}) if for all $\alpha\in I$, there exists $\beta\geq \alpha$ such that the homomorphism induced by inclusion $\pi_k(X_\alpha^H,x_0)\to \pi_k(X_\beta^H,x_0)$ (resp. the function induced by inclusion $\pi_0(X_\alpha^H)\to \pi_0(X_\beta^H)$ ) is the zero homomorphism (resp. is a constant function) for all $H\in \calF$ and for all $x_0\in X_\alpha^H$.

 \begin{theorem}\label{Browns:homotopy:Criterium}
 Let $G$ be a group and $\calF$ a family of subgroups of $G$. Let $X$ be an $\calF$-$n$-good $G$-CW-complex, and let $\{X_\alpha\}_{\alpha \in I}$ be a filtration of finite $n$-type by $G$-subcomplexes.  Consider the following statements:
 \begin{enumerate}
     \item\label{Brown:Criterium:01} $G$ is of type $\ffn$.
     \item\label{Brown:Criterium:02} For each $k<n$ we have that the filtration $\{X_\alpha\}_{\alpha \in I}$ is \esstriv{k}{\calF}.
 \end{enumerate}
 Then \eqref{Brown:Criterium:01} implies \eqref{Brown:Criterium:02}, and if, additionally, $X$ has $G$-finite $0$-skeleton then \eqref{Brown:Criterium:02} implies \eqref{Brown:Criterium:01}.
 \end{theorem}



The hypothesis in Theorem~\ref{Browns:homotopy:Criterium} that $X$  has $G$-finite $0$-skeleton in order to obtain the equivalence of (1) and (2) can be waived under assumptions on the family $\calF$, this is illustrated with Theorem~\ref{Brown:Corollary:ACC} whose proof is postponed to Section~\ref{sec:maintheorem}. Recall that a poset satisfies the ascending chain condition (ACC) if every increasing sequence  of elements is eventually constant. 

\begin{theorem}\label{Brown:Corollary:ACC} Under the assumptions of Theorem~\ref{Browns:homotopy:Criterium}, if in addition the family $\calF$ is generated by a finite collection of maximal elements, and the poset of subgroups $(\calF, \subseteq)$ satisfies the ascending chain condition, then statements~\eqref{Brown:Criterium:01} and~\eqref{Brown:Criterium:02} are equivalent. \end{theorem}

The following corollary characterizes the property $\ff{n}$ for families consisting only of finite subgroups. In particular, it provides an $\f{n}$ version of the Brown's criterion for $\fpn$ in~\cite{Br87}.

\begin{corollary}\label{Corrollary:finite:subgroups}
Under the assumptions of Theorem~\ref{Browns:homotopy:Criterium}, if in addition the family $\calF$ contains only finite subgroups of $G$, and $G$ is $\ff{0}$, then statements~\eqref{Brown:Criterium:01} and~\eqref{Brown:Criterium:02} are equivalent.
\end{corollary}
\begin{proof}
In this case, that $G$ is $\ff{0}$ implies that there is an upper bound on the cardinality of subgroups in $\calF$. Hence $(\mathcal{F}, \subseteq)$ satisfies ACC and it is generated by a finite collection of maximal elements. The statement follows from Theorem~\ref{Brown:Corollary:ACC}. 
\end{proof}

Leary and Nucinkis exhibited that finiteness properties are not preserved under finite extensions~\cite{LN03}, they found a group $G$ with a finite index subgroup $H$ such that $H$ is $\underline{\mathrm{F}}_0$ and $G$ is not $\underline{\mathrm{F}}_0$. The corollary below addresses a case where finiteness properties of a group are preserved under finite extensions.

\begin{corollary}\label{Corollary:finite:extensions}
Let $G$ be a group and let $H$ be a finite index subgroup. Let $\calF$ be a family of subgroups of $H$ which is also a family of subgroups of $G$. If $n\geq 0$ and $H$ is $\ff{n}$, then $G$ is $\ff{n}$.
\end{corollary}
\begin{proof}
That $H$ is $\ff{0}$ is equivalent to the existence of a finite collection of subgroups $\calP$ of $H$  such that any element of $\calF$ is up to conjugation a subgroup of an element of $\calP$. Hence if $H$ is $\ff{0}$ then $G$ is $\ff{0}$.

Suppose $n>0$. Let $X$ be the $n$-skeleton of a model for the classifying space $E_\calF G$. Since $H$ is $\ff{0}$, $G$ is $\ff{0}$ and hence $X$ can be assumed to have $G$-finite $0$-skeleton. The set of $G$-orbits of cells of $X$ can be well ordered in such a way that for any cell $\sigma$, its boundary intersects only cells that belong to $G$-orbits less than $G\sigma$ in the order. It follows that there is a filtration $\{X_\alpha\}_{\alpha\in I}$ of $X$ of finite $n$-type by $G$-subcomplexes. Since $H$ is finite index in $G$, $X$ is the $n$-skeleton of a model for $E_\calF H$ with $H$-finite $0$-skeleton and  $\{X_\alpha\}_{\alpha\in I}$ is a filtration of finite $n$-type by $H$-subcomplexes. Since $H$ is $\ff{n}$, \cref{Browns:homotopy:Criterium} implies that for each $k<n$ the filtration $\{X_\alpha\}_{\alpha \in I}$ is \esstriv{k}{\calF}. Then, by \cref{Browns:homotopy:Criterium} again, $G$ is $\ff{n}$.
\end{proof}

For a group  $G$ and $\calF$ the family of finite subgroups, property $\ff{n}$ is denoted by $\underline{\mathrm{F}}_n$.  Corollary~\ref{cor:Luck} below is a result of W.L\"uck that is recovered from  Corollary~\ref{Corrollary:finite:subgroups}. We recall that the analogous statement for  property $\underline{\mathrm{FP}}_n$ can be found in~\cite[Lemma~3.1]{KMPN09}. For a subgroup $H\leq G$, let $N(H)$ denote the normalizer of $H$ in $G$.

\begin{corollary}\label{cor:Luck}\cite[Theorem~4.2]{Lu00}
Let $G$ be a group and let $n\geq0$. Then $G$ is $\underline{\mathrm{F}}_n$ if and only if $G$ is $\underline{\mathrm{F}}_0$ and $N_G(H)$ is $\f{n}$ for every $H\in \calF$.
\end{corollary}

\begin{proof}
As in the proof of \Cref{Corollary:finite:extensions}, given any model $X$ for $\underline{E} G$, we can construct a filtration $\{X_\alpha\}_{\alpha \in I}$ of $X$ of finite $n$-type by $G$-subcomplexes. On the other hand we can easily verify that the normalizer $N_G(H)$ acts on $X^H$.  Moreover, $X^H$ is a model for the classifying space $\underline{E} H$ and the induced filtration $\{ X_\alpha^H \}_{\alpha \in I}$ is of finite $n$-type (see \cite[Lemma~1.3]{Lu00}). Since every finite group is of type $\f{\infty}$, we conclude that $X^H$ is $n$-good with respect to the trivial family.

Assume $G$ is $\underline{\mathrm{F}}_n$. In particular, $G$ is $\underline{\mathrm{F}}_0$.  By \Cref{Corrollary:finite:subgroups} for each $k<n$ we have that the filtration $\{X_\alpha\}_{\alpha \in I}$ is \esstriv{k}{\calF}, where $\calF$ is the family of finite subgroups. Hence, for each $H\in \calF$, $\{X_\alpha^H\}_{\alpha\in I}$ is $\pi_k$-essentially trivial for all $k<n$, thus by \Cref{Corrollary:finite:subgroups}, we have that $N_G(H)$ is $\f{n}$.

Suppose that $G$ is $\underline{\mathrm{F}}_0$ and $N_G(H)$ is $\f{n}$ for every $H\in \calF$. Let $\alpha$ be an element of $I$. Fix $H\in \calF$, then by \Cref{Corrollary:finite:subgroups},  there exists $\beta_H\geq \alpha$ (that depends on $H$ and $\alpha$) such that the homomorphism $\pi_k(X_\alpha^H)\to \pi_k(X_{\beta_H}^H)$ induced by inclusion vanishes for all $k<n$. Since $G$ is $\underline{\mathrm{F}}_0$, we know that $G$ only has finitely many conjugacy classes of finite subgroups (this is a particular feature of finite subgroups), hence the maximum $\beta$ running over $H\in \calF$ of all $\beta_H$ will be finite. Such $\beta$ will satisfy the definition of \esstriv{k}{\calF} for the filtration of $X$. Hence by \Cref{Corrollary:finite:subgroups} $G$ is $\underline{\mathrm{F}}_n$.
\end{proof}

\subsection*{Organization.}  Section~\ref{sec:preliminaries} contains preliminaries including a consequence of a result of L\"uck known as the Haeflieger Construction.  Section~\ref{sec:WitnessF0} describes a construction of witnesses of property $\ff{0}$ with particular properties, a nontrivial technique whose result is summarized as Proposition~\ref{prop:FF0andAAC}. The proofs of Theorem~\ref{Browns:homotopy:Criterium} and Theorem~\ref{Brown:Corollary:ACC} are contained in Section~\ref{sec:maintheorem}. This section also describes an alternative proof of the criterion for $\ffpn$ by Fluch and Witzel. The last section of the article contains some final remarks and a question regarding Abels's groups.

\subsection*{Acknowledgements}
 The authors thank the anonymous referee for comments and corrections, and Qayum Khan for comments.  E.M.P acknowledges funding by the Natural Sciences and Engineering Research Council of Canada, NSERC. L.J.S.S was funded by the Mexican Council of Science and Technology via the program \textit{Estancias postdoctorales en el extranjero.}

\section{Preliminaries}\label{sec:preliminaries}

\subsection{Classifying spaces for families and finiteness properties}
The following proposition can be obtained using equivariant obstruction theory. For completeness we include a proof using the definition of classifying spaces.

\begin{proposition}\label{prop:homotopy:classifying:skeletons}
Let $Y$ be an $n$-dimensional $\calF$-$G$-CW-complex and let $X$ be a model for $E_\calF G$. Then, there there exists a $G$-map $f\colon Y\to X^{(n)}$. Moreover any two $G$-maps $f_1,f_2\colon Y \to X^{(n+1)}$  are $G$-homotopic.
\end{proposition}
\begin{proof}
We know that there exists a $G$-map $f\colon Y\to X$ unique up to $G$-homotopy. Using the $G$-equivariant cellular aproximation theorem, $f$ is $G$-homotopic to a cellular $G$-map $\bar f\colon Y\to X$. Therefore the image of $\bar f$ is contained in $X^{(n)}$. 

On the other hand, given two cellular $G$ maps $f_1,f_2\colon Y \to X^{(n+1)}$, there exists a $G$-homotopy $H\colon Y\times I \to X$ between $f_1$ and $f_2$ considered as functions with codomain $X$, which is homotopic to a cellular map via a homotopy rel.  $Y\times \{0\}\cup Y\times \{1\}$. Therefore we have that $f_1,f_2$ are homotopic as functions from $Y$ to $X^{(n+1)}$.
\end{proof}

\subsection{Bredon modules}

Let $G$ be a group and let $\calF$ be a family of subgroups. The \emph{restricted orbit category} $\orf{G}$ is the category whose objects are homogenous spaces (also called orbits) $G/H$ with $H\in \calF$, and whose morphisms are $G$-maps with the canonical action of $G$ in the homogenous space $G/H$. The set of $G$-maps between the orbits $G/H$ and $G/K$ is denoted by $[G/H, G/K]_G$. 

A (contravariant) \emph{$\orf{G}$--module} is a contravariant functor $\orf{G}\to Ab$, where $Ab$ is the category of abelian groups. A \emph{morphism} $M \to N$ of $\orf{G}$--modules  is a natural transformation between the underlying functors. Since  the category of $\orf{G}$--modules is a category of functors with target an abelian category, it follows that it is an abelian category itself. In fact, for example a morphism $M\to N$ is surjective if $M(G/H)\to N(G/H)$ is surjective for all $H\in \calF$. We can also define injectiveness, exactness, etc. in an analogous way. Also the category of $\orf{G}$--modules has enough projectives. Free $\orf{G}$--modules are direct sums of modules of the form $\dbZ[-,G/H]_G$ for some $H\in\calF$. We say that a free module is finitely generated if it isomorphic to $\bigoplus_{i=0}^{m}\dbZ[-,G/H_i]_G$ with $H_1,\ldots, H_m\in \calF$. An $\orf{G}$--module $M$ is said to be finitely generated if there exists a finitely generated free $\orf{G}$--module that surjects onto $M$.

Given a $G$-CW-complex $X$ with isotropy groups on a family $\calF$ we have the Bredon cellular chain  complex, which is a chain complex of free $\orf{G}$--modules $C_i(X)$ with $i\geq 0$, such that $C_i(X)(G/H)\cong C_i(X^H)$ where the latter is the usual cellular chain group. Also $C_i(X)\cong \bigoplus \dbZ[-,G/H_i]$, where the sum runs over the set of $G$-orbits of $i$-cells and $H_i$ is the isotropy group of the corresponding $i$-cell (so that $H_i$ is only well defined up to conjugation). The \emph{Bredon homology $\orf{G}$-modules $H_*^\calF(X)$} are defined to be the homology groups of the Bredon cellular chain complex, so that $H_*^\calF(X)(G/H)=H_*(X^H)$ for all $H\in \calF$. We can analogously define the Bredon cellular chain complex (resp. Bredon homology groups) of a $G$-CW-pair $(X,A)$.

\begin{remark}
Note that these Bredon homology $\orf{G}$-modules $H_*^\calF(X)$ are different from those defined by Bredon in~\cite{Br67}. Our definition of $H_*^\calF(X)$ is a particular case of \cite[Definition 13.9]{Lu89}, which coincides with the definition of Bredon homology in~\cite{FW13}.
\end{remark}

\subsection{The Haefliger construction for families}

The following construction for the trivial family is due to Haefliger~\cite{Ha92}, which was later generalized for arbitrary families by  L\"uck~\cite[Proof~of~theorem~3.1]{Lu00}.

\begin{theorem}[The Haefliger construction for families]\label{thm:haefliger}
Let $G$ be a group. Let $\calF$ and $\calG$ be families of subgroups of $G$ such that $\calF\subseteq \calG$. Consider a $\calG$-$G$-CW-complex $X$. For each cell $\sigma$ of $X$, fix models $X_\sigma$ for $E_{\calF\cap G_\sigma} G_\sigma$. Then, for each $n\geq 0$ there exists an $\calF$-$G$-CW-complex $\hat X_n$ and a $G$-map $f_n\colon \hat X_n\to X^{(n)}$ such that:
\begin{enumerate}
    \item We have $\hat X_{n-1}\subseteq \hat X_n$ and $f_n$ restricted to $\hat X_{n-1}$ is $f_{n-1}$.
    \item For every open simplex $\sigma$ of $X^{(n)}$, $f_n^{-1}(\sigma)$ is a model for $E_{\calF\cap G_\sigma} G_\sigma$.
    \item For all $H\in \calF$, $f_n^H\colon \hat X_n^H\to (X^{(n)})^H$ is a (nonequivariant) homotopy equivalence.
    \item We have a $G$-isomorphism $f_n^{-1}(\sigma)=X_\sigma \times \sigma$.
    \end{enumerate}
\end{theorem}
\begin{proof}[Sketch of the construction proving Theorem
~\ref{thm:haefliger}]
The complexes $\hat X_n$ are constructed as follows. For each cell $\sigma$ of $X$, fixed a model $X_\sigma$ for $E_{\calF\cap G_\sigma} G_\sigma$. The construction is by induction, where $\hat X_0$ is obtained by replacing each $0$-cell $\sigma$ of $X^{(0)}$ by the chosen model $X_\sigma$ and defining $f_0\colon \hat X_0 \to X_0$ as the natural projection. Suppose that $\hat X_k$ has been constructed. Let $\sigma$ be an $k+1$ dimensional cell of $X$. Then attaching map $\varphi_\sigma\colon \partial \sigma \to X^{(n)}$ induces a map $\hat \varphi_\sigma\colon \partial( \sigma) \times X_\sigma \to \hat X_k$. Then $\hat X_{k+1}$ is obtained from $\hat X_k$ by attaching the spaces $\sigma \times X_\sigma$ via the attaching maps $\hat \varphi_\sigma$. The map $f_n\colon \hat X_n \to X^{(n)}$ is induced by the projection $(\sigma \times X_\sigma) \to \sigma$.\end{proof}

\begin{corollary}\label{reduction:F:isotropy}
Assume $X$ is  $\calF$-$n$-good. For each $p$-cell $\sigma$ of $X$, fix a $G_\sigma$-witness  $X_\sigma$ for $\ff{n-p}$. Let $\calG$ be the family of subgroups generated by the isotropy groups of $X$ and assume $\calF\subseteq\calG$.
Let $\hat X=\bigcup_{i=0}^{\infty} \hat X_i$, where $\hat X_i$ is given by Theorem~\ref{thm:haefliger} according to the chosen models $X_\sigma$. The following holds:
\begin{enumerate}
    \item  The complex $\hat X$ is $\calF$-$n$-good and all its isotropy groups belong to $\calF$. In particular $\hat X^{(n)}$ is the $n$-skeleton of a model for $E_\calF G$.
    \item If $\{X_\alpha\}_{\alpha \in I}$ is a filtration of $X$ of finite $n$-type, then $\{\hat X_\alpha\}_{\alpha \in  I}$ is a filtration of $\hat X$ of finite $n$-type.
    \item $\{X_\alpha\}_{\alpha \in I}$ is  $\pi_k$-$\calF$-essentially trivial if and only if $\{\hat X_\alpha\}_{\alpha \in I}$ is \textit{ $\pi_k$-$\calF$-essentially trivial}.
\end{enumerate}
\end{corollary}
\begin{proof}
\begin{enumerate}
    \item It follows from \cref{thm:haefliger} (3) and the fact that $X$ is $\calF$-connected up to dimension $n-1$, that $\hat X$ is $\calF$-connected up to dimension $n-1$. Given a cell $\sigma$ of $\hat X$, its isotropy group $G_\sigma$  belongs to $\calF$, therefore $\calF\cap G_\sigma$ coincides with the family of all subgroups of $G_\sigma$. Hence the one-point space is a model for $E_{\calF\cap G_\sigma}G_\sigma$, it follows that $G_\sigma$ is actually $(\calF\cap G_\sigma)$-$\mathrm F_{\infty}$. Therefore $\hat X$ is $\calF$-$n$-good.
    
    \item The $G$-maps $f_n\colon \hat X_n \to X^{(n)}$ induce a $G$-map $f\colon \hat X \to X$. Note that, from the  definition of $\hat X_n$, it follows that $\hat X_\alpha = f^{-1}(X_\alpha)$.     The hypothesis that $X_\sigma$ is a witness of property $(\calF\cap G_\sigma)$-$\mathrm F_{n-p}$ implies that $\sigma\times X_\sigma$ has $G_\sigma$-finite $n$-skeleton.     Since $f\colon \hat X_\alpha \to X_\alpha$ is $G$-equivariant cellular map, $X_\alpha$ is $G$-finite, and for each cell $\sigma$ of $X_\alpha$ the pre-image $f^{-1}(\sigma)=\sigma\times X_\sigma$ is $G_\sigma$-finite; it follows that $X_\alpha$ is $G$-finite.

    \item  Let $H\in \calF$ and $\alpha \in I$. Denote by $f_\alpha$ and $f_\beta$ be the $G$-maps $\hat X_\alpha \to X_\alpha$, $\hat X_\beta \to X_\beta$ resp. given by \cref{thm:haefliger}. For every $\beta>\alpha$, we have the following commutative diagram
    \[
    \xymatrix{\pi_k(\hat X_\alpha^H) \ar[d]_{\pi_k(f_\alpha^H)} \ar[r] & \pi_k(\hat X_\beta^H)\ar[d]^{\pi_k(f_\beta^H)} \\
    \pi_k(X_\alpha^H)\ar[r] & \pi_k( X_\beta^H) }
    \]
    where the horizontal maps are those induced by the inclusion $X_\alpha^H \subseteq X_\beta^H$, and the vertical maps are isomorphisms for all $H\in \calF$ by \cref{thm:haefliger}(3). Therefore, the upper horizontal map is the zero homomorphism if and only if the lower horizontal map is the zero homomorphism. Now the assertion follows easily.\qedhere
\end{enumerate}
\end{proof}

\section{Finding witnesses for $\ff{0}$ groups}\label{sec:WitnessF0}

Given any model $X$ for $EG$, it is always possible to collapse (all $G$-translates of) an spanning tree for the $G$-action in order to obtain a model $Y$ for $EG$ with a single $G$-orbit of $0$-cells and such that the quotient map $X\to Y$ is $G$-equivariant.

\begin{question}
Given a group $G$ and a family $\calF$ with property $\ff{0}$, and $X$ any model for $E_\calF G$, is it possible to find a  quotient $Y$ of $X$, such that $Y$ is a $G$-witness for $\ff{0}$ and the quotient map $X\to Y$ is $G$-equivariant?
\end{question}

The following proposition address in the positive the above under some assumptions on the family $\calF$. This proposition is used to prove Theorem~\ref{Brown:Corollary:ACC} in Section~\ref{sec:maintheorem}. 

\begin{proposition}\label{prop:FF0andAAC}
Let $G$ be a group,  let $\calF$ be a family of subgroups and let $X$ be a model for $E_\calF G$. 
Suppose:
\begin{enumerate}
    \item the family $\calF$ is generated by a finite collection of maximal elements, and 
    \item the poset of subgroups $(\calF, \subseteq)$ satisfies the ascending chain condition.
\end{enumerate}
Then there is a $G$-CW-complex $Y$  and a map $f\colon X\to Y$ satisfying the following conditions:
\begin{enumerate}
    \item \label{prop:ACC-01} The map $f$ is a $G$-equivariant quotient map;
    \item \label{prop:ACC-02} The $0$-dimensional skeleton $Y^{(0)}$ is $G$-finite; 
    
    \item \label{prop:ACC-02a} The isotropy of every $0$-cell of $Y$ is a maximal subgroup of $\mathcal{F}$. In particular, the isotropy of every cell of $Y$ belongs to $\mathcal{F}$.
    
    \item \label{prop:ACC-03} For each $0$-cell $y\in Y$, the preimage $f^{-1}(y)$ is a contractible subcomplex of $X^{(1)}$;
    
    \item \label{prop:ACC-04} For $k>1$, any open $k$-cell of $X$ maps homeomorphically to an open $k$-cell of $Y$.
    \item \label{prop:ACC-05} For every $H\in \calF$, the restriction $f\colon X^H \to Y^H$ is a homotopy equivalence.

\end{enumerate}
\end{proposition}

\begin{remark} \label{rem:ACC}
Recall that a group is \emph{Noetherian} if every subgroup is finitely generated, or equivalently 
any ascending chain of subgroups stabilizes after a finite length. The following statement is immediate:

Let $G$ be a group and let $\calF$ be a family of subgroups. If
\begin{enumerate}
    \item every element of $\calF$ is Noetherian, and
    \item the family is closed under unions of chains. 
\end{enumerate}
Then the poset $(\calF, \subseteq)$ satisfies the ascending chain condition.
\end{remark}


To prove Proposition~\ref{prop:FF0andAAC} we need to introduce some terminology and prove a few lemmas. For the rest of the section, let $G$ be a group,  let $\calF$ be a family of subgroups and let $X$ be a model for $E_\calF G$. 

An \emph{ascending path} in $X$ is a sequence of $0$-cells $v_0,\ldots ,v_\ell$ of $X$, and a choice of $1$-cells $e_1, \ldots , e_\ell$ such that for each $i<\ell$
\begin{enumerate}
    \item the cells $v_i$ and $v_{i+1}$ are the endpoints of the $1$-cell $e_{i+1}$ of $X$, and
    \item the isotropy of $v_i$ is a subgroup of the isotropy of $v_{i+1}$. 
\end{enumerate}
The integer $\ell$ is the \emph{length} of the path, the vertices $v_0$ and $v_\ell$ are called the \emph{initial} and \emph{terminal} points of the path respectively. The \emph{image of the path} is the subcomplex of $X$ consisting of all $0$-cells $v_i$ for $0\leq i\leq \ell$ and all $1$-cells $e_i$ for $1\leq i\leq \ell$. 

\begin{lemma}[Ascending connectivity of Fixed Point Sets of Maximal Subgroups in $\mathcal{F}$]\label{lem:AscendingPathBetweenMaximal}
Let $P$ be a maximal subgroup of $\mathcal{F}$. Then any two $0$-cells of $X^P$ are connected by an ascending path. 
\end{lemma}
\begin{proof}
Since $P$ is a maximal subgroup, any path in $X^P$ is an ascending path. Since $X$ is a model for $E_{\calF}G$, $X^P$ is path-connected, and the statement follows.
\end{proof}

\begin{lemma}[Ascending connectivity of $X$]\label{lem:AscendingPathToMaximal}
Suppose that $(\mathcal{F}, \subseteq)$ satisfies the ascending chain condition, and suppose that $\mathcal{P}$ is a collection of maximal subgroups in $\mathcal{F}$ that generates $\mathcal{F}$. Then for each $0$-cell $u$ of $X$ there is an ascending path from $u$ to a $0$-cell which has isotropy a  maximal subgroup in $\mathcal{F}$.
\end{lemma}
\begin{proof}
By the ascending chain condition, it is enough to prove that if $G_u$ is not a maximal subgroup in $\mathcal{F}$, then there is an ascending path from $u$ to a $0$-cell $v$ such that $G_u$ is a proper subgroup of $G_v$.  

Without loss of generality, assume that $G_u$ is a proper  subgroup of $P\in \mathcal{P}$. Consider the fixed point set $X^{G_u}$. Then there is path in $X^{G_u}$ from $u$ to a $0$-cell $w$ with isotropy $P$. Suppose that  sequence of $0$-cells induced by the path is $u=w_0,w_1,\dots , w_\ell=w$. Then $G_u\leq G_{w_i}$ for each $1\leq i \leq \ell$ and  $G_u\lneq G_\ell$. Let $v=w_j$ where $j=\min\{i\colon G_u \lneq G_{w_i}\}$. Observe that the sequence $u=w_0,w_1,\dots , w_j$ induces an ascending path and $G_u\lneq G_{w_j}$.
\end{proof}

A \emph{contracting tree} $T$ in $X$ is a contractible subcomplex of $X^{(1)}$ with the following properties:
\begin{enumerate}
    \item There is a $0$-cell $u$ of $T$ such that for any vertex $v$ of $T$ there is an ascending path from $v$ to $u$. A $0$-cell $u$ of $T$ with this property is called an \emph{apex} of the contracting tree;
    \item For every $0$-cell $v$ of $X$,  $T$ contains at most one vertex of the $G$-orbit of $v$.
    \item For every $g\in G$, if $T\cap gT \neq \emptyset$ then $g$ belongs to the isotropy in $X$ of an apex $u$. 
\end{enumerate}
 Note that any two apexes of a contracting tree have the same isotropy in $X$, and this subgroup shall be called the \emph{apex isotropy} of the contracting tree. 
 Observe that the third condition in the definition of contracting tree is redundant, it follows from the first two conditions. Note that a single $0$-cell is a contracting tree.

\begin{lemma}
Let $T$ be a contracting tree of $X$. Then a vertex $u$ in $T$ is an apex if and only if its stabilizer $G_u$ is the apex isotropy.
\end{lemma}
\begin{proof}
The only if part is direct from the definition. Conversely, suppose that $G_u$ is the vertex isotropy. By definition, there is an ascending path $w_0,\ldots ,w_\ell$ from $u$ to an apex $w$ of $T$. Note that the reversed sequence $w_\ell, w_{\ell-1},\cdots , w_0$ from $w$ to $u$ is an ascending path as well since each vertex $w_i$ has stabilizer equal to $G_u$. For an arbitrary vertex $v$ of $T$ there is an ascending path from $v$ to $u$ arising as the concatenation from an ascending path $v_0,\cdots , v_m=w$ from $v$ to the apex $w$, followed by the ascending path $w=w_\ell, w_{\ell-1},\cdots , w_0$.
\end{proof}

\begin{lemma}[Quotient by a Contracting Tree]\label{lem:ContractingTree}
Let $T$ be a contracting tree of $X$. Let $Y$ be the quotient of $X$ resulting from the equivalence relation  generated by 
\[ \{ (x,y)\in X\times X \colon \text{there is $g\in G$ such that $\{x,y\}\subset gT$} \} .\]
Then $Y$ admits  a $G$-CW-complex structure such that the quotient map $f\colon X\to Y$ satisfies the conclusions~\eqref{prop:ACC-01},  \eqref{prop:ACC-03},  \eqref{prop:ACC-04} and~\eqref{prop:ACC-05} of Proposition~\ref{prop:FF0andAAC}. Moreover, the image of $T$ is a $0$-cell of $Y$ with isotropy equal to the apex isotropy  of $T$.
\end{lemma}
\begin{proof} Let $P$ be an apex isotropy of $T$.  Let $T^+=\bigcup_{g\in P} gT$. 

First we prove that  $T^+$ is a contractible subcomplex of $X^{(1)}$. Note that it is enough to prove that if $T\cap gT \neq \emptyset$ then $T\cap gT$ is connected and hence contractible; indeed, this shows that the finite union of translates of $T$ is contractible, and therefore all homology groups and the fundamental group of $T^+$ are trivial by standard direct limit  arguments. Suppose that $T\cap gT\neq \emptyset$. This intersection contains a $0$-cell $u$ that is both an apex of $T$ as well as an apex of $gT$. Then for any point in $T\cap gT$, the ascending paths to $u$ in $T$ and $gT$ coincide; hence  $T\cap gT$ is connected.

Second, we argue that $T^+$ is an equivalence class. Note that all points in $T^+$ are equivalent since if $g\in P$ then $T\cap gT$ contains an apex of $T$. Conversely suppose $x\in X$ is equivalent to an element of $T^+$. Then there is a sequence  $x_0, x_1, \dots ,x_k=x$ in $X$, and a sequence $g_1, \ldots ,g_k$ in $G$ such that $\{x_{i-1},x_i\}\subset g_iT$ and $x_0\in T$. Since $x_0\in T$, it follows that $g_1\in P$. By induction $g_i\in P$ for all $i$, and hence, in particular, $x\in T^+$.

Observe that if $g\in G$ and $T^+\cap gT^+ \neq \emptyset$, then $g\in P$ and, in particular, $T^+= gT^+$.     Indeed, if $T^+\cap gT^+ \neq \emptyset$ then there are $a,b\in P$ such that $aT\cap gbT \neq \emptyset$. Therefore, by the definition of contractible tree, $a^{-1}gb \in P$ and therefore $g\in P$.

Let us observe that $Y$ carries a natural structure of $G$-CW-complex, and that the quotient map $f\colon X\to Y$ is $G$-equivariant. This is a consequence that of the following description of equivalence classes. Since $T^+$ is an equivalence class and it is a subcomplex of $X^{(1)}$, it follows that for every point $x\in X$, its equivalence class is either a single point (in the case that $x$ in the interior of a $k$-cell for $k>2$; or $x$ is in the interior of a $1$-cell whose orbit does not intersect $T$), or is of the form $gT^+$. 

The description of equivalence classes also shows that any for any $0$-cell of $Y$, the pre-image $f^{-1}(y)$ is either a $0$-cell of $X$, or of the form $gT^+$.  At this point, conclusions~\eqref{prop:ACC-01}, \eqref{prop:ACC-03}, and~\eqref{prop:ACC-04} have been proved. Note that it also follows that the image of $T^+$ in $Y$  is a $0$-cell with isotropy the apex isotropy of $T$.

It is left to prove that $f\colon X^H \to Y^H$ is a homotopy equivalence for any $H\in \mathcal{F}$. It is enough to show that for any  $H\in \calF$ the intersection  $T^+\cap X^H$ is either empty or contractible. Indeed, suppose the intersection is non-empty; since  $X^H$ contains all apexes of $T$, if there are two $0$-cells in the intersection, then the  concatenation of two ascending paths into an apex  yields a path between the two $0$-cells that is contained in $X^H$.
\end{proof}

Let $\mathcal{P}$ be a collection of subgroups generating $\calF$.   A \emph{$\mathcal{P}$-system of contracting trees} for $X$ is   collection $\{T_P \colon P \in \mathcal{P}\}$ of subspaces of $X^{(1)}$ indexed by $\mathcal{P}$ such that for every $P,Q\in \mathcal{P}$ the following properties hold:
\begin{enumerate}
    \item Each $T_P$ is a contracting tree of $X$ with apex isotropy equal to $P$; and 
    \item If $g\in G$ and $T_P \cap gT_Q \neq \emptyset$ then $P=Q$ and $g\in P$.  
\end{enumerate}

A {$\mathcal{P}$-system of contracting trees} $\{T_P \colon P \in \mathcal{P}\}$  for $X$ is a  \emph{spanning $\mathcal{P}$-system of contracting trees} if $\bigcup_{P\in \mathcal{P}} gT_P$ contains exactly one vertex of every $G$-orbit of a $0$-cell of $X$.

\begin{lemma}[Quotient by a Spanning  System of Contracting Trees]\label{lem:QuotientSpanningSystem}
Suppose that $\mathcal{P}$ is a finite collection.
Let $\{T_P \colon P \in \mathcal{P}\}$ be a spanning $\mathcal{P}$-system of contracting trees for $X$. Let $Y$ be quotient of $X$ resulting from the  equivalence relation  generated by the pairs
\[ \{ (x,y)\in X\times X \colon \text{there is $g\in G$ and $P\in \mathcal{P}$ such that $\{x,y\}\subset gT_P$} \}.\]
Then $Y$ admits a $G$-CW-complex structure such that the quotient map $X\to Y$ satisfy all the conclusions of Proposition~\ref{prop:FF0andAAC}.
\end{lemma}
\begin{proof}
Since $\mathcal{P}$ is a finite collection, conclusions~\eqref{prop:ACC-01}, \eqref{prop:ACC-03},  \eqref{prop:ACC-04} and~\eqref{prop:ACC-05}  follow from  applying Lemma~\ref{lem:ContractingTree} a finite number of times. Again, that $\mathcal{P}$ is finite, and the assumption that the  $\mathcal{P}$-system is spanning imply  conclusion~\eqref{prop:ACC-02}.  The statement~\eqref{prop:ACC-02a} follows from the ``moreover part" of Lemma~\ref{lem:ContractingTree}.
\end{proof}

In view of Lemma~\ref{lem:QuotientSpanningSystem}, Proposition~\ref{prop:FF0andAAC} follows from the existence of a spanning $\mathcal{P}$-system of contracting trees for a particular choice of a finite collection $\mathcal{P}$  of subgroups generating $\mathcal{F}$.  Specifically, 
this requires the assumptions that $\mathcal{P}$ is a minimal collection  consisting of maximal subgroups in $\mathcal{F}$, and the  hypothesis that $(\mathcal{F}, \subseteq)$ satisfies the ascending chain condition.  The existence is an application of Zorn's Lemma explained below.

\begin{lemma}[Existence of Systems of Contracting Trees]\label{lem:ExistencePSystems}
Suppose that $\mathcal{P}$ is a minimal collection of subgroups generating $\mathcal{F}$, and suppose that each $P\in \mathcal{P}$ is a maximal element of $\mathcal{F}$. Then there exists a $\mathcal{P}$-system of contracting trees.
\end{lemma}
\begin{proof}
For each $P\in \mathcal{P}$, choose a $0$-cell $v_P$ in the fixed point set $X^P$, this cell exists since $X$ is a model for $E_\mathcal{F}G$. Since $P$ is a maximal subgroup of $\mathcal{F}$ and $X$ has only isotropies in $\mathcal{F}$, it follows that the isotropy of $v_P$ equals $P$. Define $T_P$ as the single $0$-cell $v_P$.

Note that each $T_P$ is a contracting tree with apex isotropy equal to $P$. Suppose that $g\in G$, $Q,P\in \calP$  and $T_P\cap gT_Q \neq \emptyset$. This implies that $P=gQg^{-1}$. Since $\mathcal{P}$ is a minimal collection generating $\mathcal{F}$, it follows that $P=Q$ and $g\in P$. This shows that $\{T_P\colon P\in \mathcal{P}\}$ is a $\mathcal{P}$-system of contracting trees. 
\end{proof}

Let $\mathcal{A}=\{A_P \colon P \in \mathcal{P}\}$
and $\mathcal{B}=\{B_P \colon P \in \mathcal{P}\}$ be $\mathcal{P}$-systems of contracting trees for $X$. The system $\mathcal{B}$ \emph{extends} $\mathcal{A}$, denoted by $\mathcal{A} \preceq \mathcal{B}$, if for every $P\in \mathcal{P}$, $A_P$ is a subcomplex of $B_P$. This defines a partial order on the collection of
$\mathcal{P}$-systems of contracting trees for $X$. If $\calA\preceq \calB$  and $\calA \neq \calB$, we say that $\calB$ is a proper extension of $\calA$.

A $\mathcal{P}$-system of contracting trees for $X$ is called \emph{maximal} if it is maximal with respect to the partial order $\preceq$.

\begin{lemma}[Existence of Maximal Systems]\label{lem:ExistenceMaximalSystem}
Suppose that $\mathcal{P}$ is a minimal collection of subgroups generating $\mathcal{F}$, and suppose that each $P\in \mathcal{P}$ is a maximal element of $\mathcal{F}$.
There exists a maximal  $\mathcal{P}$-system of contracting trees for $X$. 
\end{lemma}
\begin{proof}
By Lemma~\ref{lem:ExistencePSystems}, there exists a  $\mathcal{P}$-system of contracting trees for $X$.  It is an observation  that every chain in the (non-empty) poset of $\mathcal{P}$-systems of contracting trees has an upper bound; namely if $\{\mathcal{A}^\alpha\colon \alpha\in I\}$ is a chain of $\mathcal{P}$-systems of contracting trees for $X$, then $\{A_P \colon P\in \mathcal{P}\}$, where $A_P=\bigcup_\alpha A^{\alpha}_P$, is $\mathcal{P}$-system of contracting trees and an upper bound of the given chain. The statement follows by  Zorn's Lemma.
\end{proof}

\begin{lemma}[Maximal implies Spanning]\label{lem:MaximalImpliesSpanning}
Suppose that $(\mathcal{F}, \subseteq)$ satisfies the ascending chain condition.
Suppose that $\mathcal{P}$ is a minimal collection of subgroups generating $\mathcal{F}$, and suppose that each $P\in \mathcal{P}$ is a maximal element of $\mathcal{F}$. 
Then every maximal  $\mathcal{P}$-system of contracting trees for $X$ is a spanning $\mathcal{P}$-system of contracting trees.
\end{lemma}
\begin{proof} We show that if a a $\mathcal{P}$-system of contracting trees is not spanning, then it not a maximal system.  Let $\{T_P \colon P \in \mathcal{P}\}$ be a $\mathcal{P}$-system of contracting trees for $X$.  Let $u$ be a $0$-cell of $X$ and suppose that $\bigcup_{P\in \mathcal{P}} T_P$ has no $0$-cell in the $G$-orbit of $u$.

\begin{step}
There exists an ascending path in $X$ with initial point a $0$-cell in the $G$-orbit of $u$ and terminal point in $T=\bigcup_{P\in \mathcal{P}} T_P$.
\end{step}
By Lemma~\ref{lem:AscendingPathToMaximal}, there is an ascending path from $u$ to a cell $v$ which has isotropy a  maximal subgroup in $\mathcal{F}$. Then $G_v=gPg^{-1}$ where $g\in G$ and $P\in \mathcal{P}$. Up to changing $u$  by a $0$-cells in its $G$-orbit, we can assume that $G_v=P$.  Let $w$ be an apex of the contracting tree $T_P$. Then, by definition, the isotropy of $w$ is $P$. By Lemma~\ref{lem:AscendingPathBetweenMaximal}, there is an ascending path from $v$ to $w$. Concatenating the ascending paths from $u$ to $v$, and the one from $v$ to $w$, the statement follows.

\begin{step}
Choose an ascending path with initial point a $0$-cell in the $G$-orbit of $u$, and terminal point in $T$. Suppose that the chosen ascending path is of minimal length among all posibilities.  Let $\gamma$ be the image of this ascending path and suppose that its  terminal point  belongs to $T_Q$ where  $Q\in \mathcal{P}$. Let 
    \[ T_Q^+ = T_Q \cup \gamma .\]
Then $\{T_P \colon P\in \mathcal{P}-\{Q\}\}\cup \{T_Q^+\}$ is a $\mathcal{P}$-system of contracting trees for $X$ that propertly extends the system 
$\{T_P \colon P\in \mathcal{P}\}$.
\end{step}

The minimality of the ascending path and the assumption that $T_Q$ is contracting tree implies that $T_Q^+$ is a contractible subcomplex of $X^{(1)}$. 
The minimality of the ascending path implies that the only $0$-cell of $\gamma$ whose orbit intersects $T$ is the terminal vertex of the path. This together with  $\gamma$ being the image of an ascending path imply that  $T^+_Q$ is a contracting tree. 
Let us verify that $\{T_P \colon P\in \mathcal{P}-\{Q\}\}\cup \{T_Q^+\}$ is a $\mathcal{P}$-system of contracting trees for $X$.
It is enough to show that if $P\in \mathcal{P}-\{Q\}$ and   $g\in G$ then  $T_P \cap gT_Q^+ = \emptyset$.  If $T_P \cap gT_Q^+ \neq \emptyset$ then $T_P\cap g\gamma \neq \emptyset$. Since $P\neq Q$, this would contradict the minimality of the ascending path with image $\gamma$.
\end{proof}

\begin{proof}[Proof of Proposition~\ref{prop:FF0andAAC}]
Let $\mathcal{P}$ be a finite and  minimal collection of subgroups generating $\mathcal{F}$, and suppose that each $P\in \mathcal{P}$ is a maximal element of $\mathcal{F}$.  By Lemmas~\ref{lem:ExistenceMaximalSystem} and~\ref{lem:MaximalImpliesSpanning}, there is a spanning $\mathcal{P}$-system of contracting trees for $X$. Then the conclusion follows by invoking Lemma~\ref{lem:QuotientSpanningSystem}.  \end{proof}

\section{Proofs of Theorems~\ref{Browns:homotopy:Criterium} and~\ref{Brown:Corollary:ACC}}\label{sec:maintheorem}

\subsection{Proof of Theorem~\ref{Browns:homotopy:Criterium}}

By \cref{reduction:F:isotropy}, replacing $X$ with $\hat X$, we can assume that $X$ is $\calF$-$n$-good, it has isotropy groups in $\calF$,  the $n$-skeleton $X^{(n)}$ is the $n$-skeleton of a model for $E_\calF G$, and $\{X_\alpha\}_{\alpha\in I}$ is a filtration of finite $n$-type by $G$-subcomplexes. As a consequence~\cref{thm:haefliger}(4), we can still 
assume that $X$ has $G$-finite $0$-skeleton. Moreover, $X$ can be assumed to be equal to its $n$-skeleton in view of \Cref{remark:Fngood}. Thus, for the rest of the proof $X$ will be assumed to be the $n$-skeleton of a model for $E_\calF G$.

Let us prove (\ref{Brown:Criterium:01})$\Longrightarrow $ (\ref{Brown:Criterium:02}). Assume $G$ is of type $\ffn$. Let $Y$ be a $G$-witness for $\ff{n}$ of dimension $n$. By~\cref{prop:homotopy:classifying:skeletons}, there exist $G$-maps $f\colon Y\to X$ and $g\colon X\to Y$. We will also denote by $g$ the restriction $g\colon X_\alpha \to Y$ for $\alpha \in I$.
Since $Y$ is $G$-finite, there exists $\alpha_0 \geq 0$ such that the image of $f$ is contained in $X_\beta$ for all $\beta\geq \alpha_0$. 

Let $\alpha\in I$ and $k<n$. Choose $\beta > \max\{\alpha, \alpha_0\}$. Then \cref{prop:homotopy:classifying:skeletons} implies that the diagram

 \[
 \xymatrix{X_\alpha^{(n-1)} \ar[d]^{g} \ar@{^{(}->}[r] & X_\beta \\
Y  \ar@{=}[r]& Y \ar[u]^f}
 \]
 is commutative up to $G$-homotopy. Taking fixed points with $H\in \calF$ in this diagram yields a (non-equivariant) commutative square up to homotopy. Hence, for all $H\in \calF$, we have a  commutative (up to an automorphism of $\pi_k( X_\beta^H)$) square
 \[
 \xymatrix{\pi_k((X_\alpha^{(n-1)})^H) \ar[r] \ar[d]^{g_*} & \pi_k( X_\beta^H) \\
\pi_k(Y^H) \ar@{=}[r]& \pi_k(Y^H)=0 \ar[u]^{f_*}}
 \]
 Therefore $\pi_k((X_\alpha^{(n-1)})^H) \to  \pi_k( X_\beta^H)$ is the zero homomorphism.
 
 Since the homomorphism $\pi_k((X_\alpha^{(n-1)})^H)\to \pi_k(X_\alpha^H)$, induced by inclusion, is an isomorphism for $k<n-1$, and is surjective for $k=n-1$, it follows that $\pi_k(X_\alpha ^H)\to \pi_k(X_\beta ^H)$ vanishes for all $k\leq n-1$ and $H\in \calF$. Thus for each $k<n$ we have that the filtration $\{X_\alpha\}_{\alpha \in I}$ is \esstriv{k}{\calF}.

\vskip 10pt

Assume $X$ as $G$-finite $0$-skeleton. We will prove (\ref{Brown:Criterium:02}) $\Longrightarrow$ (\ref{Brown:Criterium:01}).
Assume that, for each $k<n$, the filtration $\{X_\alpha\}_{\alpha \in I}$ is \esstriv{k}{\calF}. By induction one constructs $G$-witnesses $Y_i$ of $\ff{i}$ for $0\leq i\leq n$.

\textbf{i=0.} Since $X$ has finite $0$-skeleton, $Y_0=X^{(0)}$ is a witness for $\ff{0}$.

\textbf{i=1.} Since $Y_0=X^{0}$ is $G$-finite, there exists $\alpha_0$ such that $X^{(0)}$ is contained in $X_{\alpha_0}$. Consider $\alpha_1$ such that $\pi_0(X_{\alpha_0}^H)\to \pi_0(X_{\alpha_1}^H)$ is constant for every $H\in \calF$. Then any two points in the $0$-skeleton of  $X_{\alpha_0}^H$ can be connected by a path in $X_{\alpha_1}^H$, for all $H\in \calF$. Since the $0$-skeleton of $X_{\alpha_0}^H$ and $X_{\alpha_1}^H$ are equal, this implies that $X_{\alpha_1}^H$ is connected for every $H\in \calF$. Thus $Y_1=X_{\alpha_1}^{(1)}$ is a $G$-witness for $\ff{1}$.

Note if $\beta>\alpha_1$ then $X_\beta^{(1)}$ is a $G$-witness for $\ff{1}$. In particular $(X_\beta^{(1)})^H$ is connected for every $H\in \calF$, and therefore we can omit base points for higher homotopy groups of $\pi_i(X_\beta^H)$.

\textbf{i=2.}  Let $f_1\colon X^{(1)}_{\alpha_1}\to X^{(2)}$ denote the inclusion map. By~\Cref{prop:homotopy:classifying:skeletons} there is a $G$-map $g_1\colon X^{(1)} \to  X^{(1)}_{\alpha_1}$ such that $f_1\circ g_1\colon X^{(1)}\to X^{(2)}$  and the inclusion $X^{(1)}\to X^{(2)}$ are $G$-homotopic functions. Denote this $G$-homotopy by $H_1\colon X^{(1)}\times [0,1] \to X^{(2)}$ and assume that it is cellular. Therefore each $1$-cell of $X^H$ (resp. $X_\alpha^H$ with $\alpha \in I$) is homotopic to a $1$-cell of $X_{\alpha_1}^H$. 

Consider $\alpha_2$ such that $\pi_1(X_{\alpha_1}^H)\to \pi_1(X_{\alpha_2}^H)$ is trivial for every $H\in \calF$ and, in addition, $X_{\alpha_2}$ contains the image $H_1 (X^{(0)}\times [0,1])$. Observe that the second condition on the choice of $X_{\alpha_2}$ can be realized since $X^{(0)}$ is $G$-finite and hence 
$H_1 (X^{(0)}\times [0,1])$ is a $G$-finite subspace of $X^{(1)}$. 

Let $e$ be a $1$-cell of $X_{\alpha_2}$. The image $H_1(\partial(e\times[0,1]))$ is a loop in $X_{\alpha_2}$ since the image $H_1 (X^{(0)}\times [0,1])$ is contained in $X_{\alpha_2}$. Use this loop to  attach a $2$-cell to $X_{\alpha_2}^{(2)}$ with isotropy equal to the stabilizer $G_e$, and extend using the action of $G$. We can repeat the same procedure for a finite set of representatives of equivariant $1$-cells of $X_{\alpha_2}^{(1)}$ to obtain a $G$-finite $2$-dimensional $G$-CW-complex $Y_2$.

Let's verify that $Y_2$ is a $G$-witness of $\ff{2}$. 
Let $H \in \calF$.
The construction of $Y_2$ guarantees that the inclusion  $Y_2^{(1)}\to Y_2$ is $G$-homotopic to a map with image in $X_{\alpha_1}^{(1)}$. Therefore any loop in $Y_2^H$ is homotopic to a map with image in $(X_{\alpha_1}^{(1)})^H$.
On the other hand, any loop in $X_{\alpha_1}^H$ is null-homotopic in $Y_2^H$ since $\pi_1(X_{\alpha_1}^H)\to \pi_1(X_{\alpha_2}^H)$ is trivial and $X_{\alpha_2}^H$ is a subcomplex of $Y_2^H$. It follows that any loop in $Y_2^H$ is null-homotopic.

\begin{remark}\label{remark:finite:cero:skeleton}
It is worth noticing that here we are using the fact that $X$ has $G$-finite $0$-skeleton. In fact, without this assumption, we could have $0$-cells in $X_{\alpha_2}$ that do not belong to $X_{\alpha_1}$. Then the  construction does not work as we will have to add $1$-cells to $X_{\alpha_2}$ to homotope those $0$-cells into $X_{\alpha_1}$. This will create new loops that are not  homotoped into  $X_{\alpha_1}$.
\end{remark}

\textbf{i=3.}
 Consider $\alpha_3$ such that $\pi_2(X_{\alpha_2}^H)\to \pi_2(X_{\alpha_3}^H)$ is constant for every $H\in \calF$. By \cref{prop:homotopy:classifying:skeletons} we have maps $f_2\colon Y_2 \to X^{(2)}$ and $g_2\colon X^{(2)}\to Y_2$ such that the composition $g_2\circ f_2\colon Y_2^{(1)} \to Y_2$ and the inclusion map $Y_2^{(1)} \hookrightarrow Y_2$ are $G$-homotopic, via a $G$-homotopy $H_2\colon Y_2^{(1)}\times [0,1] \to Y_2$.

Analogously to the previous case, we will  attach finitely many $G$-orbits of $3$-cells to $Y_2$ in order to obtain a $G$-witness  $Y_3$ for $\ff{3}$. We will do this in three steps.

\setcounter{step}{0}
\begin{step}\label{step:extend:homotopy}
Attach finitely many equivariant $3$-cells to $Y_2$ to obtain a $\calF$-$G$-CW-complex $Z_3$ and a $G$-homotopy $H_2'\colon Y_2 \times I \to Z^3$ that extends $H_2$.
\end{step}
\begin{proof}
Let $e^2$ be a $2$-cell in $Y_2$ with attaching map $\varphi\colon S^1 \to Y_2^{(1)}$. Hence we want to extend the homotopy $H_2\colon Y_2^{(1)}\times I \to Y_2$.
 In order to do this, consider $\varphi\times Id=\varphi'\colon S^1\times I \to Y_2^{(1)}\times I$ as well as the composition $H_2\circ\varphi'\colon S^1\times I \to Y_2$. This last function can be extended to a map $\psi\colon S^2=(S^1\times I)\cup (e^2\times \{0\})\cup (e^2\times \{1\}) \to Y_2$ using $\varphi'$, the inclusion $e^2\to Y_2$, and the restriction of  $g_2 \circ f_2$ to $e_2$. We have that the $\psi$ extends to a function $D^3 \to Y_2$ if and only if $\psi$ is null-homotopic. This is not always the case, therefore we attach a $3$-cell to $Y_2$ using $\psi$ as attaching map. Now we extend using the action of $G$, and we finish after a finite number of attachings since we only have a finite number of orbits of $2$-cells in $Y_2$.
\end{proof}

\begin{step}\label{step:finitely:generated}
There are  finitely many cellular $G$-maps $\varphi_i\colon S^2\times G/K_i \to Z_3$ with the following properties:
\begin{enumerate}
    \item Each $K_i$ is in the family $\calF$;
    \item The $G$-action on the sphere $S^2$ is trivial, and diagonal on $S^2\times G/K_i$; and 
 \item The induced morphism
$\bigoplus_i H_2^\calF(S^2\times G/K_i) \to H_2^\calF(Z_3)$ is surjective.
\end{enumerate}
In particular, since $H_2^\calF(S^2\times G/K_i)  \cong  \dbZ[-, G/K_i]$, the module $H_2^\calF(Z_3)$ is a finitely generated $\orf{G}$-module.

\end{step}
\begin{proof}
We have the following diagram of $\orf{G}$-modules 
\[
\xymatrix{ & C_3(X_{\alpha_3})\cong \bigoplus\limits_{_{3-\text{cells mod }  G}}\dbZ[-,G/G_\sigma] \ar[d]^{\partial}\\
 C_2(X_{\alpha_2}^{(2)}) \ar@<1ex>[d]^{(g_2)_*}_{.}\ar[r] & C_2(X_{\alpha_3}^{(2)})\ar@<1ex>[d]^{(g_2)_*}_{.}\\
C_2(Y_2) \ar@<1ex>[u]^{(f_2)_*}_{.} \ar@{->>}[r] & C_2(Z_3) \ar@<1ex>[u]^{(f_2)_*}_{.}}
 \]
 where the horizontal arrows are the homomorphisms induced by inclusion. Let $\sigma$ be a cycle in $C_2(Z_3)$. So $\sigma$ is also a cycle in $C_2(Y_2)$. Then $(f_2)_*(\sigma)$ is a cycle in $C_2(X_{\alpha_2}^{(2)})$, and also is a cycle in $C_2(X_{\alpha_3}^{(2)})$. Since the horizontal upper map vanishes when we descend to homology, we have that $(f_2)_*(\sigma)$ is in the image of $\partial$ (which is the boundary map in the Bredon cellular chain complex). Therefore $(g_2)_*\circ(f_2)_*(\sigma)\in C_2(Z_3)$ is in the image of $(g_2)_*\circ\partial$, which is finitely generated since it is covered by a finitely generated free $\orf{G}$-module. Finally, by \cref{step:extend:homotopy}, the diagram
 \[
 \xymatrix{X_{\alpha_2^{(2)}} \ar@{^{(}->}[r]& X_{\alpha_3^{(2)}} \ar[d]^{g_2}\\
Y_2 \ar[u]^{f_2} \ar@{^{(}->}[r]& Z_3}
 \]
 is commutative up to $G$-homotopy, hence we have that $[(g_2)_*\circ(f_2)_*(\sigma)]=[\sigma]$ in $H_2^{\calF}(Z_3)$. Therefore the induced homomorphism $C_3(X_{\alpha_3})\to H_2^{\calF}(Z_3)$ is surjective.
 
 Since $X_{\alpha_3}$ is $G$-finite, it has a finite number of $G$-orbits of $3$-cells with representatives $\sigma_1,\dots, \sigma_m$. Let $K_i$ be the $G$-isotropy of $\sigma_i$. Let $\psi_i\colon S^2\times G/K_i \to Z_3$ be the attaching map of $\sigma_i \times G/K_i$ into $X_{\alpha_3}$. Then the maps $g_2\circ \psi_i\colon S^2\times G/K_i \to X_{\alpha_3}$ induced an surjective morphism 
 $ \bigoplus_i  H_2^\calF(S^2\times G/K_i) \to H_2^{\calF}(Z_3)$.
 
 We provide an sketch of this fact and leave details to the reader. Since $X$ has isotropy groups in $\calF$, it follows that $K_i$ is in $\calF$.  Since $H_2^\calF(S^2\times G/K_i)  \cong  \dbZ[-, G/K_i]$ and $C_3^{\calF}(X_{\alpha_3})\cong \bigoplus_i \dbZ[-, G/K_i]$, there is an  isomorphism 
 $ \bigoplus_i  H_2^\calF(S^2\times G/K_i) \to C_3^{\calF}(X_{\alpha_3})$. Hence the morphism $\bigoplus_i H_2^\calF(S^2\times G/K_i) \to C_3^{\calF}(X_{\alpha_3}) \to H_2^{\calF}(Z_3)$ is surjective. In fact the morphism 
 $\bigoplus_i H_2^\calF(S^2\times G/K_i) \to H_2^{\calF}(Z_3)$ is induced by the restriction maps.

\begin{step}\label{step:construct:witness}
 We can attach finitely many $G$-orbits of $3$-cells to $Z_3$ to obtain obtain a  $G$-witness $Y_3$ of $\ff{3}$.
\end{step}
\begin{proof}
  For each $i$, attach a $G$-orbit of $3$-cells to $Z_3$ via the attaching map $\varphi_i$ from the previous step to obtain a connected $G$-CW-complex $Y_3$ such that $H_j^{\calF}(Y_3)=0$ for $j=1,2$, and with finitely many $G$-orbits of cells. For all $H\in \calF$,  $Y_3^H$ is simply connected and  $H_2(Y_3^H)=0$. By Hurewicz theorem $\pi_2(Y_3^H)=0$ for all $H\in \calF$. Thus $Y_3$ is a $G$-witness for $\ff{3}$
\end{proof}

In general, for $n\geq 3$, then one obtains $Y_{n+1}$ from $Y_n$ by going through all the same steps as $Y_3$ was obtained from $Y_2$. Therefore and induction argument compludes the proof.
\end{proof}

\subsection{Proof of Theorem~\ref{Brown:Corollary:ACC}}

We need to prove that statement~\eqref{Brown:Criterium:02} implies~\eqref{Brown:Criterium:01} under the assumptions on the family $\calF$. The strategy is to modify the $G$-CW-complex $X$ so that the $0$-skeleton is $G$-finite via Proposition~\ref{prop:FF0andAAC} and then apply Theorem~\ref{Browns:homotopy:Criterium}.

Consider $Y$ the quotient complex of $X$ obtained applying~ \Cref{prop:FF0andAAC}. Then $Y$ is the $n$-skeleton of a model for $E_\calF G$ with $G$-finite $0$-skeleton. Also, the quotient of the filtration of $X$ will lead to a filtration $\{Y_\alpha\}_{\alpha \in I}$ of $Y$. It only remains to prove that $\{Y_\alpha\}_{\alpha \in I}$ is \esstriv{k}{\calF} for all $k<n$. Take a fixed $\alpha\in I$. Consider the map $\gamma \colon S^k\to Y_\alpha^H$, where $H\in \calF$. On the other hand  there exist  $G$-maps $f\colon X\to Y$ and $g\colon Y\to X$ such that the compositions are $G$-homotopic to the corresponding identity maps. Next, by $G$-compactness of $Y_\alpha$ there exists an $\alpha'\in I$ such that the image $g(Y_\alpha)$ is contained in $X_{\alpha'}$. Define $\beta'$ the index associated to $\alpha'$ via the fact that $\{X_\alpha\}$ is  \esstriv{k}{\calF}. Consider $\beta$ such that the image $f(X_{\beta'})$ is contained in $Y_\beta$. Now it is easy to see that the composition $ S^k \to Y_\alpha^H \to X_{\alpha'}^H \to X_{\beta'}^H\to  Y_\beta^H$ is null-homotopic. Since this null-homotopy does not depend on $\gamma$ (only on $k$ and $\alpha)$, then we proved that $\{Y_\alpha\}_{\alpha \in I}$ is  \esstriv{k}{\calF} for all $k<n$. 

To conclude the proof, invoke Theorem~\ref{Browns:homotopy:Criterium} with $Y$ and $\{Y_\alpha\}_{\alpha\in I}$ respectively.

\subsection{On a theorem of Fluch and Witzel} 

Brown's criterium for $\ffpn$ (main theorem of \cite{FW13}) can be also reproved following the strategy used in the proof of~\Cref{Browns:homotopy:Criterium}. For completeness we state the following theorem. The corresponding definition of \hesstriv{k}{\calF}  is the obvious one.

\begin{theorem}\label{Browns:homology:Criterium}\cite{FW13}
 Let $G$ be a group and $\calF$ a family of subgroups of $G$. Let $X$ be an $\calF$-$n$-good $G$-CW-complex, and let $\{X_\alpha\}_{\alpha \in I}$ be a filtration of finite $n$-type by $G$-subcomplexes.  Then the following statements are equivalent:
 \begin{enumerate}
     \item\label{Brown:homology:Criterium:01} $G$ is of type $\ffpn$.
     \item\label{Brown:homology:Criterium:02} For each $k<n$ we have that the filtration $\{X_\alpha\}_{\alpha \in I}$ is \hesstriv{k}{\calF}.
 \end{enumerate}
\end{theorem}

\begin{proof}[Sketch of proof]
We start applying the Haefliger construction to $X$ and we will denote $\mathbf{C}$ the augmented Bredon chain complex of $X$. Hence $\mathbf{C}$ is a $G$-witness for $\ffpn$, and we have a filtration $\{C_\alpha\}_{\alpha \in I}$, where $C_\alpha$ is the Bredon chain complex of $X_\alpha$. To prove (1) implies (2), we can apply exactly the same argument from the proof of \Cref{Browns:homotopy:Criterium}.

In order to prove (2) implies (1), we have to construct witnesses for $\ffp{i}$ for all $0\leq i\leq n$, this is, partial projective resolutions of length $i$ such that all of its projective modules are finitely generated. This can be done by induction. For the induction base, $i=0$, we proceed exactly as in \cite[Section~3]{FW13}. Now the inductive step works exactly the same as in the case $i=3$ in the proof of \Cref{Browns:homotopy:Criterium}, making the obvious changes, i.e. using $\mathbf{C}$ and $\{\mathbf{C}_\alpha\}_{\alpha \in I}$ instead of $X$ and $\{X_\alpha\}_{\alpha \in I}$, and we skip Step 3.
\end{proof}

\begin{remark}
In contrast with Theorem~\ref{Browns:homotopy:Criterium}, note that the statement of Theorem~\ref{Browns:homology:Criterium} does not include as a hypothesis the assumption that $X$ has $G$-finite $0$-skeleton.  We refer the reader to  Remark~\ref{remark:finite:cero:skeleton} on how $G$-finiteness of $X^0$ is assumed in the proof of Theorem~\ref{Browns:homotopy:Criterium}.
\end{remark}

\section{Speculation and final remarks}

For future reference we record the following corollary of the main theorem. This result provides criteria implying that a group is $G$ is of type $\ff{n-1}$, and if $n\geq 3$ then $G$ is not $\ff{n}$.  

\begin{corollary}\label{cor:fn:vs:fnplusone}
Let $G$ be a group and let $\calF$  be a family of subgroups of $G$.  Let $X$ be $G$-CW-complex such that $X^H$ is contractible for every $H\in \calF$, the stabilizer $G_\sigma$ of each cell $\sigma$ is of type  $(\calF \cap G_\sigma )$-$\f{\infty}$. Assume either 
 \begin{itemize}
     \item $X$ has $G$-finite $0$-skeleton, or
     \item the family $\calF$ is generated by a finite collection of maximal elements, and the poset of subgroups $(\calF, \subseteq)$ satisfies the ascending chain condition.
 \end{itemize}
Let $\{X_j\}_{j\geq 1}$ be a filtration such that each $X_j$ is $G$-finite.  Fix $n\geq 1$ and suppose that for all sufficiently large $j$, the complex $X_{j+1}$ is obtained from $X_j$ by the adjunction of a positive number of $G$-orbits of $n$-cells, up to $G$-homotopy. 
 Then $G$ is of type $\ff{n-1}$, and if $n\geq 3$ then $G$ is not $\ff{n}$.
\end{corollary}
\begin{proof}
Let $k< n-1$. For all $H\in \calF$,   $\pi_k(X^H)$ is trivial and 
for all sufficiently large $j$, the morphisms $\pi_k(X_j^H) \to \pi_k(X^H)=0$ induced by inclusion are isomorphisms. Therefore the filtration  $\{X_j\}_{j\geq1}$ is $\pi_k$--$\calF$-essentially trivial for all $k<n-1$. The hypotheses imply that $X$ is $\calF$-$(n-1)$-good. Using the hypothesis,  \Cref{Browns:homotopy:Criterium} or \Cref{Brown:Corollary:ACC}  implies that $G$ is $\ff{n-1}$.

For the negative part, we use Brown's argument~\cite[Proof of Cor. 3.3]{Br87} which shows that  $\{H_{n-1}(X_j)\}$ is not essentially trivial under the given hypotheses. Since $n\geq 3$, Hurewicz theorem implies that $\{\pi_{n-1}(X_j)\}$ is not essentially trivial. As a consequence $\{X_j\}$ is not \esstriv{n}{\calF}. Since $X$ is $\calF$-$n$-good and has $G$-finite $0$-skeleton, \Cref{Browns:homotopy:Criterium} implies  $G$ is not of type $\ff{n}$.
\end{proof}

The statement of Corollary~\ref{cor:fn:vs:fnplusone} for the trivial family is a result of Brown~\cite[Corollary 3.3]{Br87}. He used the criterion to prove that if $p$ is an odd prime and $\Gamma_n < \mathbb{GL}_{n+1}(\mathbb{Z}[1/p])$ is the group of upper triangular matrices whose extremal diagonal entries are $1$, then $\Gamma_n$ is $\mathrm{FP}_{n-1}$ but not $\mathrm{FP}_{n}$. The result for $n\leq 4$ was known before Brown's work, see~\cite{Br87} for references. This particular class of groups is known as Abels's groups.  

The finiteness properties phenomena of Abels's groups were revisited and generalized by Witzel~\cite{Wi13} for a larger class of groups. Specifically the main result in~\cite{Wi13} states that for a fixed odd prime $p$ and integers for $0<m\leq n$,  there is a group $\Gamma_{m,n}$ of $(n+1)\times(n+1)$ upper triangular matrices with coefficients $\Z[1/p]$ and subject to certain technical conditions such that $\Gamma_{m,n}$ is $\mathrm{F}_{n-1}$ and $\underline{\mathrm{F}}_{m-1}$ but not $\mathrm{F}_{n}$ nor $\underline{\mathrm{F}}_m$. The Abels's group $\Gamma_n$ can be taken as $\Gamma_{n,0}$, see~\cite[Example 1.1]{Wi13}.

Witzel's strategy includes a detailed study of the action of $\Gamma$ on its (extended) Bruhat-Tits building $X$, and in particular, the structure of the fixed-point sets $X^H$ of finite subgroups $H$ of $\Gamma$. The argument also uses  a version of Brown's criterion~\cite{FW13} for Bredon homology which characterizes properties $\underline{\mathrm{FP}}_m$ in terms of the connectivity properties of $X$, a result analogous to~\Cref{Browns:homotopy:Criterium}. In particular, the  argument that  $\Gamma_{m,n}$ is $\underline{\mathrm{F}}_{m-1}$ consists of invoking the mentioned criterion to obtain  $\underline{\mathrm{FP}}_{m-1}$, and then to directly argue that the group is $\underline{\mathrm{F}}_2$. A slight simplification of the argument can be obtained by invoking~\Cref{Browns:homotopy:Criterium} to obtain $\underline{\mathrm{F}}_{m-1}$, avoiding the verification of $\underline{\mathrm{F}}_2$.

The family of isotropies of $X$, with respect to the action of $\Gamma$ is strictly bigger than the family of finite subgroups, and every isotropy group is a finitely generated nilpotent group. It is possible that the class of Abels's groups exhibit the finiteness phenomena as suggested in the following general question.

\begin{question}\label{question:finiteness:abels}
Let $0<m\leq n$. Let $\calG$ be the family of subgroups of $\Gamma_{m,n}$ generated by the the isotropy groups of $X$.  If $\calF$ is a subfamily of $\calG$, is there a positive number $r$ such that $\Gamma_{m,n}$ is $\ff{r-1}$ and it is not $\ff{r}$?
\end{question}

The following lemma was pointed out by the referee of the article.
From here on we follow the notation of the preceding question.

\begin{lemma}
Abels's group $\Gamma_{m,n}$ is not of type $\ff{\infty}$.
 \end{lemma}
\begin{proof}
We proceed by contradiction. Assume  $\Gamma_{m,n}$ is of type $\ff{\infty}$. Let $Y$ be  a model for $E_{\calF}\Gamma_{m,n}$ that witness $\ff{\infty}$. Recall that virtually finitely generated nilpotent groups are of type $\mathrm{F}_\infty$, this follows from~\cite[p.~213]{Br87} and that $\mathrm{F}_\infty$ is preserve under finite extensions~\cite[p.~197]{Br87}. Hence every element of $\calF$ is virtually finitely generated nilpotent (see~\cite[proof~of~theorem~B(b)]{BA85}). Let $\hat Y$ be the space obtained from~\Cref{thm:haefliger} applied to $Y$ and the trivial subfamily of $\mathcal{F}$. Then $\hat Y$ is a $\Gamma_{m,n}$-witness for $\mathrm{F}_\infty$ which contradicts Witzel's result~\cite[Theorem 3.9]{Wi13}. 
\end{proof}

\begin{corollary}\label{cor:partial:answer}
If $\Gamma_{m,n}$ is $\ff{0}$, then  there is a positive number $r$ such that $\Gamma_{m,n}$ is $\ff{r-1}$ and it is not $\ff{r}$.
\end{corollary}
The authors do not know whether $\Gamma_{m,n}$ is $\ff{0}$ for $\calF$ the family of isotropy groups of the $\Gamma_{m,n}$-action on the Bruhat-Tits building $X$.

We expect that in particular cases the first failure of finiteness properties with respect to families could be computed is some of these groups. Consider the group $\Gamma_n$ as defined by Brown in~\cite[Section~6]{Br87}, $X$ its Bruhat-Tits buidling, and $\calG$ the family generated by the isotropy groups of $X$.

\begin{corollary}
If the family $\calG$ is a closed under unions, then $\Gamma_n$ is of type $\gf{n-1}$ but not of type $\gf{n}$.
\end{corollary}
\begin{proof}
The statement is a consequence of \Cref{cor:fn:vs:fnplusone}. We verify the hypotheses of the corollary for $\Gamma=\Gamma_n$, the building $X$, and the family $\calG$:
\begin{itemize}
    \item For each $H\in \calG$, the fixed point set $X^H$ is nonempty. Indeed, since $X$ is a $\CAT(0)$-space and the action of $\Gamma_n$ in $X$ is by isometries, then $X^H$ is convex, in particular is contractible.
    \item For each cell $\sigma$ of $X$, the family $\calG\cap \Gamma_\sigma$ of subgroups of $\Gamma_\sigma$ coincides with the family of all subgroups of $\Gamma_\sigma$. Therefore $\Gamma_\sigma$ is of type $\calG\cap \Gamma_\sigma$-$\mathrm{F}_\infty$.
    \item Every group in $\calG$ is  finitely generated nilpotent, therefore is Noetherian. Additionally, by hypothesis, $\calG$ is closed under unions. Hence by \Cref{rem:ACC}, $\calG$ satisfies the ascending chain condition.
    \item Brown exhibits a filtration $\{X_j\}_{j\geq 1}$  such that each $X_j$ is $\Gamma$-finite. In addition, Brown proves that for all sufficiently large $j$, the complex $X_{j+1}$ is obtained from $X_j$ by the adjunction of a positive number of $\Gamma$-orbits of $n$-cells, up to $\Gamma$-homotopy (see \cite[Lemma~6.2]{Br87}).\qedhere
\end{itemize}
\end{proof}
In the case of the group $\Gamma_n$, we do not know whether the family $\calG$ satisfies the ascending chain condition, or equivalently, closed under unions. 


 \bibliographystyle{plain}

\end{document}